\documentclass[conference]{ieeeconf}
\IEEEoverridecommandlockouts
\usepackage{epsfig} 
\usepackage{times} 
\usepackage{amsmath} 
\usepackage{amssymb}  
\usepackage{url}
\usepackage{enumerate}
\usepackage[final]{changes}
\usepackage{cite}
\usepackage{color}
\usepackage{xspace}
\usepackage{algorithmic}

\usepackage[ruled]{algorithm}
\graphicspath{{./figures/}}
\usepackage{graphicx}
\usepackage{subfigure}

\usepackage{epstopdf}
\usepackage{float}
\DeclareGraphicsExtensions{.png}
\usepackage{comment}
\usepackage{caption}

\newtheorem{lemma}{Lemma}
\newtheorem{theorem}{Theorem}

\newtheorem{assumption}{Assumption}

\newtheorem{example}{Example}

\def\prox{\operatorname{prox}}
\def\R{\mathbb{R}}
\def\N{\mathbb{N}}

\title{\LARGE \bf Optimal convergence rates of \replaced{totally asynchronous optimization}{asynchronous optimization methods}}

\author{Xuyang Wu, Sindri Magn\'{u}sson, Hamid Reza Feyzmahdavian, and Mikael Johansson\thanks{X. Wu and M. Johansson are with the Division of Decision and Control Systems, School of Electrical Engineering and Computer Science, KTH Royal Institute of Technology, SE-100 44 Stockholm, Sweden. Email: {\tt {\{xuyangw,mikaelj\}@kth.se}.}}
\thanks{S. Magn\'{u}sson is with the Department of Computer and System Science, Stockholm University, SE-164 07 Stockholm, Sweden. Email: {\tt sindri.magnusson@dsv.su.se}.} \thanks{H. R. Feyzmahdavian is with ABB Cooperate Research, SE-721 78 V\"{a}ster\r{a}s, Sweden. Email: {\tt hamid.feyzmahdavian@se.abb.com}}\thanks{This work was supported in part by the funding from Digital Futures and in part by the Swedish Research Council (Vetenskapsr\r{a}det) under grant 2020-03607.}}

\begin{document}

	\maketitle

	\begin{abstract}
	Asynchronous optimization algorithms are at the core of modern machine learning and resource allocation systems. However, most convergence results consider bounded information delays and several important algorithms lack guarantees when they operate under total asynchrony. In this paper, we derive explicit convergence rates for the proximal incremental aggregated gradient (PIAG) and the asynchronous block-coordinate descent (Async-BCD) methods under a specific model of total asynchrony, and show that the derived rates are order-optimal. The convergence bounds provide an insightful understanding of how the growth rate of the delays deteriorates the convergence times of the algorithms. Our theoretical findings are demonstrated by a numerical example.
	\end{abstract}

	%
	%
\section{Introduction}
\if
Story idea:
\begin{itemize}
    \item Why asynchronous iterations/optimization. Cite representative work to motivate current interest in the area (both from CDC and elsewhere)
    \item Bertskas and Tsitsiklis. Restricted class of iteration converges also under unbounded asynchrony. Modern optimization algorithms rarely satisfy these conditions.
    \item A lot of results for bounded delays (cite a few papers), but only very few results are available when the upper bound is not known, or when it may grow without a fixed bound. Cite a few works, discuss their shortcomings
    \item Explain our key results. Why are they challenging, interesting, useful.
\end{itemize}
\fi
 Distributed and parallel algorithms are powerful tools for solving large-scale problems. These algorithms coordinate multiple computing nodes to solve the overall problem. The coordination can be synchronous, meaning that each node needs to wait for all other nodes to successfully conclude their computations and communications before proceeding to the next iteration. This is clearly inefficient: the slowest node dictates the convergence speed, systems become sensitive to single node failures, and the implementation overhead for synchronization can be large.  Therefore,  asynchronous algorithms that need no synchronization are often preferred \cite{iutzeler2013asynchronous,zhang2019asyspa,assran2020advances}. However, compared to synchronous algorithms, asynchronous algorithms are more difficult to analyze, and their convergence properties are not as well understood. 

Early efforts on establishing convergence properties of asynchronous algorithms  were made in the 1980s by Bertsekas and Tsitsiklis, e.g.,~\cite{bertsekas1983distributed,tsitsiklis1986distributed,bertsekas1989convergence}. They considered two models for asynchrony, partial asynchrony (``bounded delays'') and total asynchrony (``unbounded delays''), and analyzed convergence for certain classes of algorithms under these two models. However, these algorithm classes do not cover many modern optimization algorithms, such AsySPA \cite{zhang2019asyspa}, PIAG \cite{aytekin16}, Async-BCD \cite{liu2014}, Arock \cite{Peng16}, and Asynchronous SGD \cite{recht2011hogwild}.

In the last decade, the convergence of many modern algorithms is established under partial asynchrony, e.g., AsySPA \cite{zhang2019asyspa}, PIAG \cite{aytekin16,vanli2018global,Sun19}, Async-BCD \cite{liu2014,Davis16,Sun17}, ARock \cite{Peng16,Feyzmahdavian21}, and Asynchronous SGD \cite{recht2011hogwild}. However, only a few algorithms are shown to work under total asynchrony~\cite{ren2020delay,zhou2021distributed,mishchenko2018delay}. In particular, \cite{ren2020delay,zhou2021distributed} study Asynchronous SGD and \cite{mishchenko2018delay} focuses on a delay-tolerant averaged proximal gradient algorithm.
 None of these papers cover total asynchrony for PIAG and Async-BCD, which are the focus of our work.
 Moreover, different from us, the works \cite{ren2020delay,zhou2021distributed,mishchenko2018delay} do not characterize how the unbounded delay affects convergence rates or explores the existence of an optimal rate.
 %

This paper studies two asynchronous optimization algorithms, PIAG and Async-BCD, under total asynchrony. None of these algorithms has been proven to converge under total asynchrony before.
We make the following contributions:
\begin{itemize}
    \item We derive explicit convergence rates for PIAG and Async-BCD under a model of total asynchrony.
    \item We prove that the derived convergence rates for the two methods are optimal in terms of order.
    \item We use the convergence bounds to provide insight and understanding of how the growth rate of delays slows down the convergence of PIAG and Async-BCD.
\end{itemize}

\subsection*{Notation and Preliminaries} We use $\mathbb{N}$ and $\mathbb{N}_0$ to denote the set of natural numbers and the set of natural numbers including zero, respectively. We let $[m] = \{1,\ldots,m\}$ for any $m\in\mathbb{N}$ and represent $x\in\mathbb{R}^d$ as $x=(x^{(1)}, \ldots,x^{(m)})$, where each $x^{(i)}\in\mathbb{R}^{d^{(i)}}$ and $\sum_{i=1}^n d^{(i)} = d$. We define the proximal operator of a function $r:\mathbb{R}^d\rightarrow\mathbb{R}\cup\{+\infty\}$ as
\begin{equation*}
	\operatorname{prox}_{r}(x) = \underset{y\in\mathbb{R}^d}{\operatorname{\arg\!\min}} ~r(y)+\frac{1}{2}\|y-x\|^2.
\end{equation*}
We say a differentiable function $f:\R^d\rightarrow \R$ is $L$-smooth if
\begin{equation*}
	\|\nabla f(x)-\nabla f(x+h)\|\le L\|h\|,~\forall x,h\in\mathbb{R}^d.
\end{equation*}
We call $f$ is $\hat{L}$-block-wise smooth with respect to a partition $x=(x^{(1)}, \ldots,x^{(m)})$ if for all $i,j\in[m]$, and $h^{(j)}\in\R^{d^{(j)}}$,
\begin{equation}\label{eq:coordinatesmooth}
	\|\nabla_i f(x+U^{(j)}h^{(j)}) - \nabla_i f(x)\|\le \hat{L}\|h^{(j)}\|.
\end{equation}
Here, $\nabla_j f(\cdot)$ is the partial gradient of $f$ with respect to the $j$th block and $U^{(j)}:\mathbb{R}^{d^{(j)}}\rightarrow\mathbb{R}^d$ maps $h^{(j)}\in\mathbb{R}^{d^{(j)}}$ into a $d$-dimensional vector where the $j$th block is $h^{(j)}$ and other blocks are zero. For an $L$-smooth function $f$ and a convex function $r:\R^d\rightarrow \R\cup\{+\infty\}$, we say  $P(x)=f(x)+r(x)$ satisfies the proximal PL condition \cite{karimi16} for $\sigma>0$ if
\begin{equation}\label{eq:proximalPLcond}
	\sigma(P(x)-P^\star)\le -L\hat{P}(x),~\forall x\in\operatorname{dom}(P),
\end{equation}
where $P^\star=\min_{x\in\mathbb{R}^d} P(x)$ and 
\begin{equation*}
\hat{P}(x)=\min_{y\in\mathbb{R}^d}~\{\langle \nabla f(x), y-x\rangle+\frac{L}{2}\|y-x\|^2+r(y)-r(x)\}.
\end{equation*}

\section{Problem Statement}
\label{sec:probformuandtrans}
\replaced{We focus on optimization problems on the form}{
This paper considers the following problem:}
\begin{equation}\label{eq:prob}
	\underset{x\in\mathbb{R}^d}{\min}~ P(x) = f(x)+r(x),
\end{equation}
where $f:\mathbb{R}^d\rightarrow\mathbb{R}$ is smooth and possibly non-convex, and $r:\mathbb{R}^d\rightarrow\mathbb{R}\cup\{+\infty\}$ is convex \replaced{but possibly}{and may be} non-differentiable. Such a composite structure is common \deleted{in optimization, e.g.,} in\added{, for example,} machine learning where $f$ is a loss and $r$ is a regularizer. The function $r$ can also \replaced{represent the}{include the} indicator function of a convex set.


We \replaced{consider}{use} the proximal incremental aggregated gradient (PIAG) algorithm \cite{aytekin16} and the asynchronous block-coordinate descent (Async-BCD) method \cite{liu2014} to solve \eqref{eq:prob}.

\subsection{PIAG}
The algorithm solves problem \eqref{eq:prob} for $f$ on the form
\begin{align*}
    f(x) &= \frac{1}{n}\sum_{i=1}^n f^{(i)}(x) 
\end{align*}
using the following update
\begin{align}
	g_k &= \frac{1}{n}\sum_{i=1}^n \nabla f^{(i)}(x_{k-\tau_k^{(i)}}),\label{eq:PIAGgkcompute}\\
	x_{k+1} &= \operatorname{prox}_{\gamma_k r}(x_k-\gamma_k g_k),\label{eq:PIAGxupdate}
\end{align}
where $k$ is the iteration index, $\tau_k^{(i)}\in [0,k]$ is the delay of the gradient $\nabla f^{(i)}$ at iteration $k$, and $\gamma_k\ge 0$ is the step-size. The update \eqref{eq:PIAGgkcompute}--\eqref{eq:PIAGxupdate} is often implemented in the parameter server architecture \cite{LiM13}\deleted{including one master and $n$ workers}, where each worker $i$ computes $\nabla f^{(i)}(x_{k-\tau_k^{(i)}})$ and the master aggregates all the \added{most recent} local gradients to form (\ref{eq:PIAGgkcompute}) and updates the iterate using \eqref{eq:PIAGxupdate}. The implementation of PIAG is detailed in Algorithm \ref{alg:PIAG}.

\begin{algorithm}[tb]
	\caption{PIAG \cite{aytekin16,vanli2018global}}
	\label{alg:PIAG}
	\begin{algorithmic}[1]
		\STATE {\bfseries Input}: initial iterate $x_0$, number of iteration $k_{\max}\in\N$.
		\STATE {\bfseries Initialization:}
		\STATE The master sets $k\leftarrow0$, $g^{(i)}\leftarrow\nabla f^{(i)}(x_0)$ $\forall i\in [n]$, and $g_0\leftarrow\frac{1}{n}\sum_{i=1}^n \nabla f^{(i)}(x_0)$.
		\WHILE{$k\le  k_{\max}$: each worker $i\in [n]$ \emph{asynchronously} and \emph{continuously}}
		\STATE receive $x_k$ from the master.
		\STATE compute $\nabla f^{(i)}(x_k)$.
		\STATE send $\nabla f^{(i)}(x_k)$ to the master.
		\ENDWHILE
		\WHILE{$k\le  k_{\max}$: the master}
		\STATE Wait until a set ${\mathcal R}$ of workers return.
		\FOR{all $w\in {\mathcal R}$}
		\STATE update $g^{(w)} \leftarrow  \nabla f^{(w)}(x_l)$.
		\ENDFOR
		\STATE set $g_k \leftarrow \frac{1}{n}\sum_{i=1}^n g^{(i)}$.
		\STATE determine the step-size $\gamma_k$.
		\STATE update $x_{k+1}\leftarrow\operatorname{prox}_{\gamma_kr}(x_k-\gamma_k g_k)$.
		\STATE set $k\leftarrow k+1$.
		\FOR{all $w\in {\mathcal R}$}
		\STATE send $x_{k}$ to worker $w$.
		\ENDFOR
		\ENDWHILE
	\end{algorithmic}
\end{algorithm}

\subsection{Async-BCD}\label{ssec:problembarriertran}

Suppose that the non-smooth function $r$ in problem \eqref{eq:prob} is separable, i.e., for a partition $x=(x^{(1)}, \ldots, x^{(m)})$ with $x^{(i)}\in\mathbb{R}^{d^{(i)}}$ and $\sum_{i=1}^m d^{(i)}=d$, it holds that $r(x) = \sum_{i=1}^m r^{(i)}(x^{(i)})$ $\forall x\in\R^d$.
When the dimension $d$ of $x$ is large, one \replaced{attractive}{typical} method for solving \eqref{eq:prob} is the block-coordinate descent (BCD) method: at each $k\in\mathbb{N}_0$, randomly choose $j\in [m]$ and execute the update
\begin{equation*}
	x_{k+1}^{(j)} = \prox_{\gamma_k r^{(j)}}(x_{k}^{(j)}-\gamma_k\nabla_j f(x_k)),
\end{equation*}
where \deleted{$k\in\N_0$ is the iteration index,}  $\nabla_j f(\cdot)$ is the partial gradient of $f$ with respect to the $j$th block $x^{(j)}$ and 
$\gamma_k\ge 0$ is the step-size.
Async-BCD implements BCD using multiple \replaced{processors in a}{servers in the} shared memory setting \cite{Peng16}\replaced{. The}{: the} decision vector is stored in shared memory and at each iteration $k$, one worker $i_k\in [n]$ updates
\begin{align}\label{eq:asyncbcd}
	x_{k+1}^{(j)} = \prox_{\gamma_k r^{(j)}}(x_{k}^{(j)}-\gamma_k\nabla_j f(x_{k-\tau_k})).
\end{align}
Here, $x_{k-\tau_k}$ is the decision vector that worker $i_k$ has read from shared memory and based its partial gradient computation on. The delay $\tau_k$ measures the number of updates that other processors have performed between the read and write operations of worker $i_k$. The block index $j$ is drawn by $i_k$ uniformly at random at time $k-\tau_k$. Algorithm \ref{alg:asyncbcd} details the implementation of Async-BCD.

\begin{algorithm}[tb]
	\caption{Async-BCD}
	\label{alg:asyncbcd}
	\begin{algorithmic}[1]
		\STATE {\bfseries Setup:} initial iterate $x_0$, number of iteration $k_{\max}\in\N$.
		\WHILE{$k\le  k_{\max}$: each worker $i\in [n]$ \emph{asynchronously} and \emph{continuously}}
		\STATE sample $j\in [m]$ uniformly at random.
		\STATE compute $\nabla_j f(\hat{x}_k)$ based on $\hat{x}_k$ read at time $k-\tau_k$.
		\STATE determine the step-size $\gamma_k$.
		\STATE compute $x_{k+1}^{(j)}$ by \eqref{eq:asyncbcd}.
		\STATE write on the shared memory.
		\STATE set $k\leftarrow k+1$.
		\STATE read $x_k$ from the shared memory.
		\ENDWHILE
	\end{algorithmic}
\end{algorithm}

\section{Main Result}\label{sec:result}

\replaced{In this section, we will derive convergence results for PIAG and Async-BCD under a totally asynchronous delay model. In our setting, the totally asynchronous model of Bertsekas and Tsitsiklis~\cite{bertsekas2015parallel} would allow the delays to grow unbounded, as long as no processor ceases to update and}{
This section derives convergence results for PIAG and Async-BCD under total asynchrony \cite{bertsekas2015parallel} which allows the delays to be unbounded but requires}
\begin{equation}\label{eq:asyeq}
    \lim_{k\rightarrow+\infty} k-\tau_k = +\infty,
\end{equation}
where $\tau_k=\max_{i\in [n]} \tau_k^{(i)}$ for PIAG. Convergence \emph{rate} results under this model are unlikely since \replaced{it does not impose any strict bound on how quickly the delays can grow. We thus focus on a particular model of asynchrony that satisfies (\ref{eq:asyeq}).}{it assumes nothing on growing speed of delays. We instead consider a particular asynchrony model which specializes \eqref{eq:asyeq}.}
\begin{assumption}\label{asm:asynchrony}
	For some $a\in(0,1)$, $b\in [0,1]$, and $c\ge 0$, \begin{equation*}
	    \tau_k\le \min(k,ak^b+c),~\forall k\in\N_0.
	\end{equation*}
\end{assumption}
Clearly, Assumption \ref{asm:asynchrony} guarantees \eqref{eq:asyeq}. Moreover, by varying $b\in[0,1]$ we can move seamlessly between several interesting and important models of asynchrony. In particular,
\begin{itemize}
    \item $b=0$ yields bounded delays: $\tau_k\le \min(k,a+c)$.
    \item $b=1/2$ is sublinear growth: $\tau_k\le \min(k,a\sqrt{k}+c)$.
    \item $b=1$ is a linear delay bound: $\tau_k\le \min(k,ak+c)$.
\end{itemize}
The linearly growing delay bound is the largest polynomial growth we can have because when $b>1$, the total asynchrony condition \eqref{eq:asyeq} no longer holds. Our analysis may be extended to other delay models, e.g., those in \cite{Feyzmahdavian21} and \cite{chen2007global}.

\subsection{PIAG}\label{ssec:PIAGconv}
Let us first present the convergence rate guarantees for PIAG under the delays characterized by Assumption~\ref{asm:asynchrony}. For convenient notation, we introduce
\begin{equation*}
	\phi(k) = 
	\begin{cases}
		k^{1-b}, & b\in [0,1),\\
		\ln k, & b=1,
	\end{cases}\quad\forall k\in\N_0,
\end{equation*}
where $b$ is the delay bound parameter in Assumption \ref{asm:asynchrony}.
\begin{theorem}\label{theo:piag}
	Suppose that each  $f^{(i)}$ is $L_i$-smooth, $r$ is convex and closed, $P^\star:=\min_x P(x)>-\infty$, and Assumption \ref{asm:asynchrony} holds. Define $L=\sqrt{({1}/{n})\sum_{i=1}^n L_i^2}$. Let $\{x_k\}$ be generated by the PIAG algorithm with
	\begin{equation}\label{eq:PIAGstep}
		\gamma_k =\frac{h}{L(a(\frac{k+c}{1-a})^b+c+1)},~\forall k\in\mathbb{N}_0,
	\end{equation}
	where $h\in(0,1)$. Then,
	\begin{enumerate}[(i)]
		\item There exist $\xi_k\in \partial r(x_k)$ $\forall k\in\N_0$ such that
			\begin{equation*}
			    \min_{t\le k} \|\nabla f(x_t)+\xi_t\|^2=O(1/\phi(k)).
			\end{equation*}
		\item If each $f^{(i)}$ is convex, then
		\begin{equation*}
    		P(x_k)-P^\star=O(1/\phi(k)).
		\end{equation*}
		\item If $P$ satisfies the proximal PL-condition \eqref{eq:proximalPLcond}, then \begin{equation*}
		   P(x_k)-P^\star=O(\lambda^{\phi(k)})
		\end{equation*}
		for some $\lambda\in (0,1)$.
	\end{enumerate}
\end{theorem}
\begin{proof}
	See Appendix \ref{ssec:proofpiag}.
\end{proof}
Table \ref{tab:relation} extracts the relationship between the delay bound, admissible step-size, and convergence rates in Theorem \ref{theo:piag}.
\begin{table}[!htb]
    \centering
    \begin{tabular}{c|c|c|c}
        \hline
        delay bound & step-size & rate  & rate\\
        & & (non-convex, convex) & (proximal PL)\\
        \hline
        $O(k^b)$, $b<1$ & $O(k^{-b})$ & $O(1/k^{1-b})$ & $O(\lambda^{k^{1-b}})$\\
        \hline
        $O(k^b)$, $b=1$ & $O(1/k)$ & $O(1/\ln k)$ & $O(1/k)$\\
        \hline
    \end{tabular}
    \caption{asynchrony, step-size, and convergence rate.}
    \label{tab:relation}
\end{table}

Note that when $b=0$, which corresponds to bounded delays, the convergence rates in Table \ref{tab:relation} \replaced{match}{reach} those of PIAG under partial asynchrony \cite{aytekin16,vanli2018global,Sun19,Wu22} and those of gradient descent \cite{nesterov2003introductory,karimi16}, i.e., $O(1/k)$ for non-convex and convex objective functions and linear convergence if the proximal PL-condition holds. The table quantifies how large delays limit the admissible step-sizes and deteriorate the convergence rates, which agrees with intuition.
\subsection{Async-BCD}\label{ssec:Asyncbcdconv}
Based on the block-wise smoothness assumption \eqref{eq:coordinatesmooth}, the following theorem establishes \replaced{convergence rates for}{the convergence of} Async-BCD in solving problem \eqref{eq:prob}.

\begin{theorem}\label{theo:coor}
	Suppose that $f$ is $\hat{L}$-block-wise smooth with respect to the partition $x=(x^{(1)}, \ldots, x^{(m)})$, each $r^{(i)}$ is convex  and closed,  $P^\star:=\min_x P(x)>-\infty$, and \added{that} Assumption~\ref{asm:asynchrony} holds. Let $\{x_k\}$ be generated by the Async-BCD algorithm with
	\begin{equation}\label{eq:asyncbcdstep}
		\gamma_k =\frac{h}{\hat{L}(a(\frac{k+c}{1-a})^b+c+1)},~\forall k\in\mathbb{N}_0,
	\end{equation}
	where $h\in(0,1)$. Then,
	\begin{equation*}
		\min_{t\le k} E[\|\tilde{\nabla} P(x_t)\|^2]=O(1/\phi(k)),
	\end{equation*}
	where $\tilde{\nabla} P(x_t)=\hat{L}(\operatorname{prox}_{\frac{1}{\hat{L}} r}(x_t-\frac{1}{\hat{L}} \nabla f(x_t))-x_t)$.
	\end{theorem}
	\begin{proof}
		See Appendix \ref{ssec:proofasyncbcd}.
	\end{proof}
	In Theorem \ref{theo:coor}, $\tilde{\nabla} P(x)=\mathbf{0}$ if and only if $\mathbf{0}\in \partial P(x)$, i.e., $x$ is a stationary point of problem \eqref{eq:prob}. When $b=0$, our convergence rate is of the same order compared to Async-BCD under partial asynchrony \cite{Davis16,Sun17}. The relationship between delay bound, step-size, and convergence rate of Async-BCD \replaced{is summarized in}{ can also be illustrated by} Table~\ref{tab:relation} for non-convex objective functions. Again, a larger delay requires smaller step-sizes and leads to a slower convergence.

\subsection{Optimal convergence rate}

\replaced{The next result establishes that the convergence rates in the preceding theorems are optimal under  Assumption~\ref{asm:asynchrony}, and not a consequence of the particular step-size policies.}{
This subsection shows that the rates in Theorems \ref{theo:piag}--\ref{theo:coor} are optimal for both PIAG and Async-BCD under the asynchrony model defined by Assumption \ref{asm:asynchrony}.}
\begin{theorem}\label{theo:convergencelimit}
    \replaced{The convergence rates for PIAG and Async-BCD derived in Theorems~\ref{theo:piag}--\ref{theo:coor} are order-optimal.}{In terms of order, the rates in Theorems \ref{theo:piag}--\ref{theo:coor} are optimal for PIAG and Async-BCD, respectively.}
\end{theorem}

\subsubsection{Proof of Theorem \ref{theo:convergencelimit}}

\replaced{We prove our claim by constructing an objective function and a delay sequence that satisfy the assumptions of the preceding theorems, and are such that the proposed rates are optimal.
}{
Our main idea is to prove that for an example satisfying the conditions in Theorems \ref{theo:piag}--\ref{theo:coor}, the two methods can only be guaranteed to converge at the rates in Theorems \ref{theo:piag}-\ref{theo:coor}, respectively.}

Let $r\equiv 0$ and let $f$ be $L$-smooth for some $L>0$. Then, both PIAG and Async-BCD reduce to
\begin{equation}\label{eq:asyncgd}
	x_{k+1} = x_k - \gamma_k \nabla f(x_{k-\tau_k}).
\end{equation}
Now, consider the delay sequence
\begin{equation}\label{eq:delayexample}
    \tau_k = k - T_t, \quad \text{if } k\in [T_t, T_{t+1}),
\end{equation}
where $\{T_t\}$ is defined by $T_0=0$ and 
\begin{equation}\label{eq:Gammatplus1}
    T_{t+1} = \max\{\kappa \in\N_0: \kappa - (a\kappa^b+c)\le T_t\}+1
\end{equation}
for some $a\in(0,1)$, $b\in[0,1]$, and $c\ge 0$. In this way, for any $k\in [T_t, T_{t+1})$, it holds that $k-(ak^b+c)\leq T_t$. By (\ref{eq:delayexample}), $\tau_k\leq ak^b+c$ and $\tau_k\leq k$, so $\{\tau_k\}$ satisfies Assumption~\ref{asm:asynchrony}.

By substituting \eqref{eq:delayexample} into \eqref{eq:asyncgd}, we obtain
\begin{equation*}
	x_{k+1} = x_{T_t} - (\sum_{\ell =T_t}^k \gamma_\ell) \nabla f(x_{T_t}),~~\forall k\in[T_t, T_{t+1}).
\end{equation*}
This implies that $x_k$, $k\in\mathbb{N}$ is obtained by performing $\max\{t\in\N_0: T_t\le k-1\}+1$ steps of gradient descent starting from $x_0$. Moreover, we prove in Appendix \ref{ssec:prooflemmaseq} that
\begin{equation}\label{eq:maximaliter}
    \max\{t\in\N_0:\; T_t\le k-1\}+1=O(\phi(k)).
\end{equation}
The result now follows by observing that after $O(\phi(k))$ steps of gradient descent on a general $L$-smooth function or a general $L-$smooth and proximal PL function, we cannot obtain better order of convergence \cite{drori2014performance,karimi16} than those in Theorems \ref{theo:piag}--\ref{theo:coor}.
\section{Numerical Experiments}\label{sec:numericalexample}

	We demonstrate the theoretical results in Theorems \ref{theo:piag}--\ref{theo:coor} and \replaced{evaluate}{test} the practical performance of the two methods under total asynchrony \replaced{in simulations}{by simulation}. We consider \added{a binary} classification problem on the training data set of RCV1 \cite{lewis2004rcv1} using the regularized logistic regression model:
	\begin{align*}
	    f(x) &= \frac{1}{N}\sum_{i=1}^N\left(\log(1+e^{-p_i(q_i^Tx)})+\frac{\lambda_2}{2}\|x\|^2\right),\\
	    r(x) &= \lambda_1\|x\|_1,
	\end{align*}
	where $p_i$ is the feature of the $i$th sample, $q_i$ is the corresponding label, and $N$ is the number of samples. We use $(\lambda_1,\lambda_2)=(10^{-5},10^{-4})$ in all simulations.

	\subsection{PIAG}\label{ssec:PIAGsimu}
	
	We split the $N$ samples into $n=10$ batches and assign each batch to a single worker. We consider the following delay model: $\tau_0^{(i)}=0$ for all $i\in[n]$. For all $k\in\N$ and $i\in[n]$, if $\tau_{k-1}^{(i)}\le \min(k,ak^b+c)-1$, then $\tau_k^{(i)}=\tau_{k-1}^{(i)}+1$; Otherwise, $\tau_k^{(i)}$ is randomly drawn from $[\min(k,\lfloor ak^b+c\rfloor)]$. We use $a=0.1$ and $c=0$, and consider $b=0.2$, $0.6$ and $1$ to evaluate the effect of delays. Note that the constructed delay sequence satisfies Assumption \ref{asm:asynchrony}.
	
	We plot the objective error $P(x_k)-P^\star$ generated by PIAG in Fig \ref{fig:PIAGpractical}, and the theoretical bound $O(1/\phi(k))$ in Theorem \ref{theo:piag} for the convex objective functions in Fig \ref{fig:PIAGtheo}, where the exact value of the bound is obtained by substituting \eqref{eq:summationgammabound} into \eqref{eq:appendPIAGconvex}. Although the rate $O(\lambda^{\phi(k)})$ in Theorem \ref{theo:piag} also holds since the objective function satisfies the proximal PL-condition, it is quite slow because the parameter $\lambda_2$ is small.
	Observe from Fig \ref{fig:PIAGpractical} that PIAG tends to converge for all three $b$'s, and the convergence speed deteriorates as $b$ increases. These demonstrate Theorem \ref{theo:piag}. Through comparison between Fig \ref{fig:PIAGpractical}--\ref{fig:PIAGtheo} that are with different $y$-scale, although the theoretical bound is much larger than the practical objective error for all three values of $b$, their decreasing speed are similar.
	\begin{figure}[!htb]
		\centering
		\subfigure[practical]{
			\includegraphics[scale=0.8]{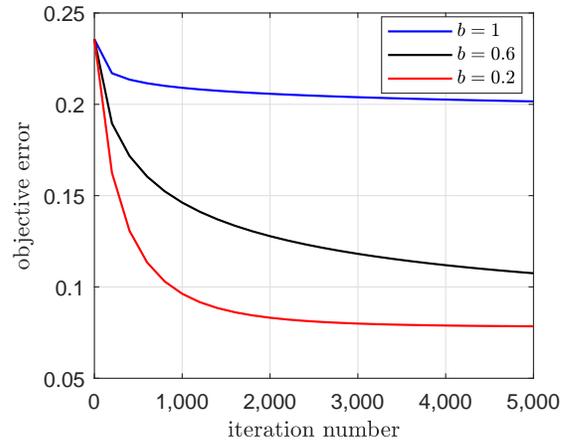}\label{fig:PIAGpractical}}
		\vfill
		\vspace{-0.25cm}
		\subfigure[theoretical]{
			\includegraphics[scale=0.8]{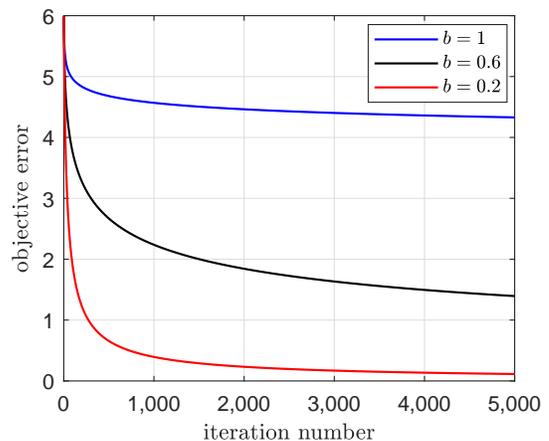}\label{fig:PIAGtheo}}
		\caption{Convergence of PIAG}
		\label{fig:PIAG}
	\end{figure}

	\subsection{Async-BCD}
	
	We use $n=8$ \replaced{processors}{servers} and split the \replaced{decision vector}{variable} $x$ evenly into $m=14$ blocks. We set $\tau_0=0$. For all $k\in\N$, $\tau_k = \tau_{k-1}+1$ if $\tau_{k-1}+1\le \min(k,\lfloor ak^b+c\rfloor)$ and is randomly drawn from $[\min(k,\lfloor ak^b+c\rfloor)]$ otherwise. Like above, we set $a=0.1$ and $c=0$, and consider $b=0.2$, $0.6$, and $1$. The resulting delay sequence satisfies Assumption \ref{asm:asynchrony}.

    Fig \ref{fig:bcd} plots the convergence of objective error $P(x_k)-P^\star$ generated by Async-BCD. We observe that for small value $0.2$ of $b$, the convergence to optimum is clear and fast and for larger values $0.6$ and $1$, the convergence to optimum is hard to observe which is normal because the delays corresponding to $b=0.6$ and $b=1$ increase too fast and coordinate-type methods often converge slowly in terms of iteration number.
	\begin{figure}
        \centering
        \includegraphics[scale=0.8]{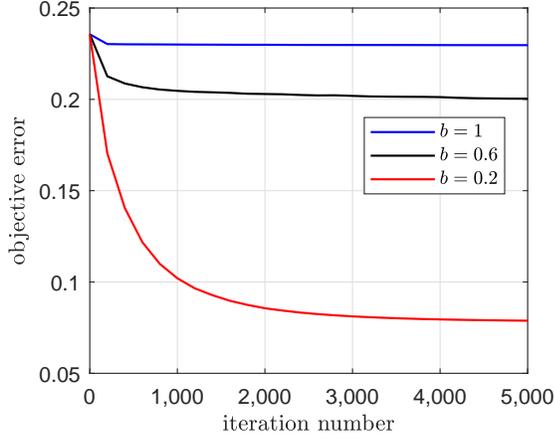}
        \caption{Convergence of Async-BCD}
        \label{fig:bcd}
    \end{figure}

\section{Conclusion}\label{sec:conclusion}

We have derived explicit convergence rates of PIAG and Async-BCD under a model of computation that allows for a broad range of totally asynchronous behaviours. The convergence rates are optimal in terms of the order of iteration index $k$ and reflect how asynchrony affects the convergence times of the algorithms. The theoretical results were validated in simulations. We believe that the proposed techniques apply also to other asynchronous optimization algorithms, but leave such studies for future work.

\appendix

\subsection{Proof of Theorem \ref{theo:piag}}\label{ssec:proofpiag}

The proof uses \cite[Theorem 2]{Wu22}.
	\begin{theorem}[\cite{Wu22}]\label{theo:PIAGinICML}
    Under the conditions in Theorem \ref{theo:piag}, if for some $h\in(0,1)$,
	\begin{equation}\label{eq:piaggeneralstep}
	    \sum_{t=k-\tau_k}^k \gamma_t\le \frac{h}{L},
	\end{equation}
	then
	\begin{enumerate}[(1)]
			\item There exist $\xi_k\in \partial r(x_k)$ $\forall k\in\mathbb{N}_0$ such that	\begin{equation*}
				\!\!\!\sum_{k=1}^\infty\!\gamma_{k-1}\|\nabla f(x_k)\!+\!\xi_k\|^2\!\le\!\frac{2(h^2\!-\!h\!+\!1)(P(x_0)\!-\!P^\star)}{1-h}.
			\end{equation*}
			\item If each $f^{(i)}$ is convex, then
			\begin{equation}\label{eq:appendPIAGconvex}
				P(x_k)-P^\star\le \frac{P(x_0)-P^\star+\frac{1}{2a_0}\|x_0-x^\star\|^2}{1+\frac{1}{a_0}\sum_{t=0}^{k-1}\gamma_t},
			\end{equation}
			where $a_0=\frac{h(h+1)}{L(1-h)}$.
			\item If $P$ satisfies the proximal PL-condition \eqref{eq:proximalPLcond}, then
			\begin{equation*}
				P(x_k)\!-\!P(x^\star)\!\le\! e^{-\frac{3\beta\sigma(1-\tilde{h})}{4(\tilde{h}^2-\tilde{h}+1)}\sum_{t=0}^{k-1}\gamma_t}(P(x_0)\!-\!P^\star),
			\end{equation*}
			where $\tilde{h} = \frac{1+h}{2}$ and $\beta = \min\left(1,\frac{1-h}{2h}\frac{L}{\sigma}\right)$.
		\end{enumerate}
	\end{theorem}
	
	To prove Theorem \ref{theo:piag} using Theorem \ref{theo:PIAGinICML}, we first show that $\{\gamma_k\}$ in \eqref{eq:PIAGstep} satisfies \eqref{eq:piaggeneralstep}. Because $\tau_k\le ak^b+c\le ak+c$, for any $t\in [k-\tau_k, k]$, we have $t\ge k-\tau_k\ge (1-a)k-c$, so that
	\begin{equation*}
	    \gamma_t = \frac{h}{L(a(\frac{t+c}{1-a})^b+c+1)} \le \frac{h}{L(ak^b+c+1)}.
	\end{equation*}
	Using the above equation and $\tau_k\le ak^b+c$, we have
    \begin{equation*}
		\sum_{t=k-\tau_k}^k \gamma_t\le \frac{(\tau_k+1)h}{L(ak^b+c+1)}\le \frac{h}{L},
	\end{equation*}
	i.e., \eqref{eq:piaggeneralstep} holds.
	
	Next, we show that $1/\sum_{t=0}^k \gamma_t = O(1/\phi(k))$. For any $t\in\mathbb{N}$,
	\begin{equation*}
		\begin{split}
			\gamma_t &= \frac{h}{L(a(\frac{t+c}{1-a})^b+c+1)}\\
			&= \frac{h(t+c)^{-b}}{L(a(1-a)^{-b}+(c+1)(t+c)^{-b})}\\
			&\ge\frac{h(t+c)^{-b}}{L(a(1-a)^{-b}+(c+1)^{1-b})}.
		\end{split}
	\end{equation*}
	In addition, $\gamma_0=\frac{h}{L(a(\frac{c}{1-a})^b+c+1)}$. Then, for any $k\in\N$,
	\begin{equation}\label{eq:summationgammabound}
		\begin{split}
		\sum_{t=0}^{k-1} \gamma_t&\ge 
			\gamma_0+\sum_{t=1}^{k-1} \gamma_t\\
			&\ge \gamma_0+\int_1^k \frac{h(s+c)^{-b}}{L(a(1-a)^{-b}+(c+1)^{1-b})} ds\\
			&=\frac{h}{L(a(\frac{c}{1-a})^b+c+1)}\\
			&+\begin{cases}
		    \frac{h((k+c)^{1-b}-(1+c)^{1-b})}{L(a(1-a)^{-b}\!+\!(c+1)^{1-b})(1-b)}, & b\!\in\! [0,1),\\
				\frac{h\ln(\frac{k+c}{1+c})}{L(a/(1-a)+1)}, & b=1,
			\end{cases}
		\end{split}
	\end{equation}
	which indicates $1/\sum_{t=0}^{k-1} \gamma_t = O(1/\phi(k))$. Hence, the results in Theorem \ref{theo:piag} for convex and proximal PL functions hold.
	
	Also note that the result in Theorem \ref{theo:PIAGinICML} for non-convex objective functions implies
	\begin{equation}\label{eq:runningbestpiag}
	\begin{split}
	    &(\sum_{t=0}^{k-1}\gamma_t)\min_{t\le k}\|\nabla f(x_t)\!+\!\xi_t\|^2\\
	    \le& \sum_{t=1}^k\gamma_{t-1}\|\nabla f(x_t)+\xi_t\|^2\\
	    \le& \sum_{t=1}^\infty\gamma_{t-1}\|\nabla f(x_t)+\xi_t\|^2\\
	    \le&\frac{2(h^2\!-\!h\!+\!1)(P(x_0)\!-\!P^\star)}{1-h}.
	\end{split}
	\end{equation}
	By substituting \eqref{eq:summationgammabound} into \eqref{eq:runningbestpiag}, we obtain the result in Theorem \ref{theo:piag} for non-convex objective functions.

    \subsection{Proof of Theorem \ref{theo:coor}}\label{ssec:proofasyncbcd}
    The proof uses \cite[Theorem 3]{Wu22}.
    \begin{theorem}[\cite{Wu22}]
        Under the conditions in Theorem \ref{theo:coor}, if for some $h\in(0,1)$,
	\begin{equation*}
	    \sum_{t=k-\tau_k}^k \gamma_t\le \frac{h}{\hat{L}},\quad\forall k\in\N_0,
	\end{equation*}
	then
	\begin{equation*}
		\sum_{k=0}^\infty \gamma_kE[\|\tilde{\nabla} P(x_k)\|^2]\le \frac{4m(P(x_0)-P^\star)}{1-h}.
	\end{equation*}
    \end{theorem}
    Using similar derivation of \eqref{eq:runningbestpiag} in Appendix \ref{ssec:proofpiag}, we have
	\begin{equation*}
		\min_{t\le k}E[\|\tilde{\nabla} P(x_t)\|^2]\le \frac{4m(P(x_0)-P^\star)}{(1-h)\sum_{t=0}^{k-1} \gamma_t}.
	\end{equation*}
    Moreover, similar to the derivation of \eqref{eq:summationgammabound}, $1/\sum_{t=0}^k \gamma_t = O(1/\phi(k))$. Then, we obtain the result.
\subsection{Proof of \eqref{eq:maximaliter}}\label{ssec:prooflemmaseq}

By \eqref{eq:Gammatplus1}, for any $\kappa>T_{t+1}-1$, we have $\kappa - (a\kappa^b+c)>T_t$. Therefore,
\begin{equation}\label{eq:Gammatlowerbound}
    \begin{split}
        T_{t+1} &\ge T_t + aT_{t+1}^b+c\\
        &\ge T_t + aT_t^b+c,
    \end{split}
\end{equation}
where the second step uses $T_{t+1}\ge T_t$ derived from the first step.

\textbf{Case 1: $b\in [0,1)$.} The proof uses induction. Let $\eta=a(1-b)2^{-\frac{b}{1-b}}$. Suppose that $T_t\ge (\eta t)^{\frac{1}{1-b}}$ for some $t\in\N_0$, which holds naturally at $t=0,1$. Then, by \eqref{eq:Gammatlowerbound},
\begin{equation}\label{eq:Gammatdefinevt}
    \begin{split}
            & T_{t+1}-(\eta (t+1))^{\frac{1}{1-b}}\\
        \ge & (\eta t)^{\frac{1}{1-b}}+a(\eta t)^{\frac{b}{1-b}}+c-(\eta (t+1))^{\frac{1}{1-b}}\\
        =   & (\eta t)^{\frac{1}{1-b}}\underbrace{\left(1+\frac{a}{\eta t}-(1+\frac{1}{t})^{\frac{1}{1-b}}\right)}_{v(t)}+c.
    \end{split}
\end{equation}
Note that $v'(t) = -\frac{1}{t^2}(\frac{a}{\eta}-\frac{1}{1-b}(1+\frac{1}{t})^{\frac{b}{1-b}})$, which satisfies $v'(t)\le-\frac{1}{t^2}(\frac{a}{\eta}-\frac{1}{1-b}2^{\frac{b}{1-b}})=0$ when $t\ge 1$. Hence, $v(t)$ is monotonically decreasing on $[1,+\infty)$ and $v(t)\ge \lim_{\ell\rightarrow +\infty} v(\ell)=0$ for all $t\ge 1$, which, together with \eqref{eq:Gammatdefinevt}, gives $T_{t+1}\ge (\eta(t+1))^{\frac{1}{1-b}}$. Hence, $t\le \frac{(T_t)^{1-b}}{\eta}$ for all $t\in\N_0$ and $\max\{t\in\N_0:T_t\le k-1\}\le \frac{(k-1)^{1-b}}{\eta}$, i.e., $\max\{t\in\N_0:T_t\le k-1\}+1= O(k^{1-b})$.

\textbf{Case 2: $b=1$.} By \eqref{eq:Gammatlowerbound},
\begin{equation*}
    T_{t+1}\ge (1+a)T_t \ge (1+a)^{t}T_1\ge (1+a)^t,
\end{equation*}
so that $t\le \ln\frac{T_t}{1+a}+1$. Hence, $\max\{t\in\N_0:T_t\le k-1\}\le \ln\frac{k-1}{1+a}+1$, i.e., $\max\{t\in\N_0:T_t\le k-1\}+1=O(\ln k)$.

Concluding the two cases, we complete the proof.

\bibliographystyle{IEEEtran}
\bibliography{reference}

\begin{thebibliography}{10}
\providecommand{\url}[1]{#1}
\csname url@samestyle\endcsname
\providecommand{\newblock}{\relax}
\providecommand{\bibinfo}[2]{#2}
\providecommand{\BIBentrySTDinterwordspacing}{\spaceskip=0pt\relax}
\providecommand{\BIBentryALTinterwordstretchfactor}{4}
\providecommand{\BIBentryALTinterwordspacing}{\spaceskip=\fontdimen2\font plus
\BIBentryALTinterwordstretchfactor\fontdimen3\font minus
  \fontdimen4\font\relax}
\providecommand{\BIBforeignlanguage}[2]{{%
\expandafter\ifx\csname l@#1\endcsname\relax
\typeout{** WARNING: IEEEtran.bst: No hyphenation pattern has been}%
\typeout{** loaded for the language `#1'. Using the pattern for}%
\typeout{** the default language instead.}%
\else
\language=\csname l@#1\endcsname
\fi
#2}}
\providecommand{\BIBdecl}{\relax}
\BIBdecl

\bibitem{iutzeler2013asynchronous}
F.~Iutzeler, P.~Bianchi, P.~Ciblat, and W.~Hachem, ``Asynchronous distributed
  optimization using a randomized alternating direction method of
  multipliers,'' in \emph{52nd IEEE conference on decision and control}.\hskip
  1em plus 0.5em minus 0.4em\relax IEEE, 2013, pp. 3671--3676.

\bibitem{zhang2019asyspa}
J.~Zhang and K.~You, ``{AsySPA}: An exact asynchronous algorithm for convex
  optimization over digraphs,'' \emph{IEEE Transactions on Automatic Control},
  vol.~65, no.~6, pp. 2494--2509, 2019.

\bibitem{assran2020advances}
M.~Assran, A.~Aytekin, H.~R. Feyzmahdavian, M.~Johansson, and M.~G. Rabbat,
  ``Advances in asynchronous parallel and distributed optimization,''
  \emph{Proceedings of the IEEE}, vol. 108, no.~11, pp. 2013--2031, 2020.

\bibitem{bertsekas1983distributed}
D.~P. Bertsekas, ``Distributed asynchronous computation of fixed points,''
  \emph{Mathematical Programming}, vol.~27, no.~1, pp. 107--120, 1983.

\bibitem{tsitsiklis1986distributed}
J.~Tsitsiklis, D.~Bertsekas, and M.~Athans, ``Distributed asynchronous
  deterministic and stochastic gradient optimization algorithms,'' \emph{IEEE
  transactions on automatic control}, vol.~31, no.~9, pp. 803--812, 1986.

\bibitem{bertsekas1989convergence}
D.~P. Bertsekas and J.~N. Tsitsiklis, ``Convergence rate and termination of
  asynchronous iterative algorithms,'' in \emph{Proceedings of the 3rd
  International Conference on Supercomputing}, 1989, pp. 461--470.

\bibitem{aytekin16}
A.~Aytekin, H.~R. Feyzmahdavian, and M.~Johansson, ``Analysis and
  implementation of an asynchronous optimization algorithm for the parameter
  server,'' \emph{arXiv preprint arXiv:1610.05507}, 2016.

\bibitem{liu2014}
J.~Liu, S.~Wright, C.~R{\'e}, V.~Bittorf, and S.~Sridhar, ``An asynchronous
  parallel stochastic coordinate descent algorithm,'' in \emph{International
  Conference on Machine Learning}.\hskip 1em plus 0.5em minus 0.4em\relax PMLR,
  2014, pp. 469--477.

\bibitem{Peng16}
Z.~Peng, Y.~Xu, M.~Yan, and W.~Yin, ``{AR}ock: an algorithmic framework for
  asynchronous parallel coordinate updates,'' \emph{SIAM Journal on Scientific
  Computing}, vol.~38, no.~5, pp. A2851--A2879, 2016.

\bibitem{recht2011hogwild}
B.~Recht, C.~Re, S.~Wright, and F.~Niu, ``Hogwild!: A lock-free approach to
  parallelizing stochastic gradient descent,'' \emph{Advances in Neural
  Information Processing Systems}, vol.~24, pp. 693--701, 2011.

\bibitem{vanli2018global}
N.~D. Vanli, M.~Gurbuzbalaban, and A.~Ozdaglar, ``Global convergence rate of
  proximal incremental aggregated gradient methods,'' \emph{SIAM Journal on
  Optimization}, vol.~28, no.~2, pp. 1282--1300, 2018.

\bibitem{Sun19}
T.~Sun, Y.~Sun, D.~Li, and Q.~Liao, ``General proximal incremental aggregated
  gradient algorithms: Better and novel results under general scheme,''
  \emph{Advances in Neural Information Processing Systems}, vol.~32, pp.
  996--1006, 2019.

\bibitem{Davis16}
D.~Davis, ``The asynchronous palm algorithm for nonsmooth nonconvex problems,''
  \emph{arXiv preprint arXiv:1604.00526}, 2016.

\bibitem{Sun17}
T.~Sun, R.~Hannah, and W.~Yin, ``Asynchronous coordinate descent under more
  realistic assumption,'' in \emph{Proceedings of the 31st International
  Conference on Neural Information Processing Systems}, 2017, pp. 6183--6191.

\bibitem{Feyzmahdavian21}
H.~R. Feyzmahdavian and M.~Johansson, ``Asynchronous iterations in
  optimization: New sequence results and sharper algorithmic guarantees,''
  \emph{arXiv preprint arXiv:2109.04522}, 2021.

\bibitem{ren2020delay}
Z.~Ren, Z.~Zhou, L.~Qiu, A.~Deshpande, and J.~Kalagnanam, ``Delay-adaptive
  distributed stochastic optimization,'' in \emph{Proceedings of the AAAI
  Conference on Artificial Intelligence}, vol.~34, no.~04, 2020, pp.
  5503--5510.

\bibitem{zhou2021distributed}
Z.~Zhou, P.~Mertikopoulos, N.~Bambos, P.~Glynn, and Y.~Ye, ``Distributed
  stochastic optimization with large delays,'' \emph{Mathematics of Operations
  Research}, 2021.

\bibitem{mishchenko2018delay}
K.~Mishchenko, F.~Iutzeler, J.~Malick, and M.-R. Amini, ``A delay-tolerant
  proximal-gradient algorithm for distributed learning,'' in
  \emph{International Conference on Machine Learning}.\hskip 1em plus 0.5em
  minus 0.4em\relax PMLR, 2018, pp. 3587--3595.

\bibitem{karimi16}
H.~Karimi, J.~Nutini, and M.~Schmidt, ``Linear convergence of gradient and
  proximal-gradient methods under the polyak-{\l}ojasiewicz condition,'' in
  \emph{Joint European Conference on Machine Learning and Knowledge Discovery
  in Databases}.\hskip 1em plus 0.5em minus 0.4em\relax Springer, 2016, pp.
  795--811.

\bibitem{LiM13}
M.~Li, L.~Zhou, Z.~Yang, A.~Li, F.~Xia, D.~G. Andersen, and A.~Smola,
  ``Parameter server for distributed machine learning,'' in \emph{Big Learning
  NIPS Workshop}, vol.~6, 2013.

\bibitem{bertsekas2015parallel}
D.~Bertsekas and J.~Tsitsiklis, \emph{Parallel and distributed computation:
  numerical methods}.\hskip 1em plus 0.5em minus 0.4em\relax Athena Scientific,
  2015.

\bibitem{chen2007global}
T.~Chen and L.~Wang, ``Global $\mu$-stability of delayed neural networks with
  unbounded time-varying delays,'' \emph{IEEE Transactions on Neural Networks},
  vol.~18, no.~6, pp. 1836--1840, 2007.

\bibitem{Wu22}
X.~Wu, S.~Magn\'{u}sson, H.~R. Feyzmahdavian, and M.~Johansson,
  ``Delay-adaptive step-sizes for asynchronous learning,'' \emph{arXiv preprint
  arXiv:2202.08550}, 2022.

\bibitem{nesterov2003introductory}
Y.~Nesterov, \emph{Introductory lectures on convex optimization: A basic
  course}.\hskip 1em plus 0.5em minus 0.4em\relax Springer Science \& Business
  Media, 2003, vol.~87.

\bibitem{drori2014performance}
Y.~Drori and M.~Teboulle, ``Performance of first-order methods for smooth
  convex minimization: a novel approach,'' \emph{Mathematical Programming},
  vol. 145, no.~1, pp. 451--482, 2014.

\bibitem{lewis2004rcv1}
D.~D. Lewis, Y.~Yang, T.~Russell-Rose, and F.~Li, ``{RCV1}: A new benchmark
  collection for text categorization research,'' \emph{Journal of machine
  learning research}, vol.~5, no. Apr, pp. 361--397, 2004.

\end{thebibliography}

\end{document}